\documentclass[10pt]{amsart}
\usepackage{amssymb}
\usepackage{bm}
\usepackage{amsfonts}
\usepackage{amssymb,amsfonts,amsmath,amsthm,graphicx}
\usepackage{comment}
\usepackage{slashbox,mathrsfs}
\usepackage[all,poly]{xy}
\usepackage{float}
\usepackage{extarrows}
\usepackage{multirow}

\allowdisplaybreaks

\begin{document}
\theoremstyle{plain}
\newtheorem{thm}{Theorem}[section]
\newtheorem{theorem}[thm]{Theorem}
\newtheorem*{theorem2}{Theorem}
\newtheorem{lemma}[thm]{Lemma}
\newtheorem{corollary}[thm]{Corollary}
\newtheorem{corollary*}[thm]{Corollary*}
\newtheorem{proposition}[thm]{Proposition}
\newtheorem{proposition*}[thm]{Proposition*}
\newtheorem{conjecture}[thm]{Conjecture}
\theoremstyle{definition}
\newtheorem{construction}[thm]{Construction}
\newtheorem{notations}[thm]{Notations}
\newtheorem{question}[thm]{Question}
\newtheorem{problem}[thm]{Problem}
\newtheorem{remark}[thm]{Remark}
\newtheorem{remarks}[thm]{Remarks}
\newtheorem{definition}[thm]{Definition}
\newtheorem{claim}[thm]{Claim}
\newtheorem{assumption}[thm]{Assumption}
\newtheorem{assumptions}[thm]{Assumptions}
\newtheorem{properties}[thm]{Properties}
\newtheorem{example}[thm]{Example}
\newtheorem{comments}[thm]{Comments}
\newtheorem{blank}[thm]{}
\newtheorem{observation}[thm]{Observation}
\newtheorem{defn-thm}[thm]{Definition-Theorem}

\def\supp{\operatorname{supp}}


\title[Weyl invariant polynomial and deformation quantization]{Weyl invariant polynomial and deformation quantization on K\"ahler manifolds}

\author{Hao Xu}
        \address{Center of Mathematical Sciences, Zhejiang University, Hangzhou, Zhejiang 310027, China;
        Department of Mathematics, Harvard University, Cambridge, MA 02138, USA}
        \email{haoxu@math.harvard.edu}

        \begin{abstract}
        Given a polynomial $P$ of partial derivatives of the K\"ahler metric, expressed as a linear combination of directed multigraphs, we prove a simple criterion
        in terms of the coefficients for $P$ to be an invariant polynomial, i.e. invariant under the transformation of coordinates. As applications, we
        prove an explicit composition formula for covariant differential operators under a canonical basis, also known as invariant differential operators in the case of bounded symmetric domains. We also prove a general explicit
        formula of star products on K\"ahler manifolds.
        \end{abstract}


    \maketitle

\section{Introduction}

Let $ds^2=g_{ij}dx^i dx^j$ be the metric tensor in a local frame on a Riemannian manifold. Consider the algebra generated by
partial derivatives of the metric $\{g_{ij\,\alpha}\}_{|\alpha|\geq1}$,
it can be shown that all polynomials in the variables $\{g_{ij\,\alpha}\}$ invariant under
coordinate transformations,
arise from complete contractions of covariant differentiations of curvature tensors.
The proof requires H. Weyl's classical invariant theory for orthogonal groups by restriction
to the normal coordinate systems. Weyl's invariant theory also played an important role in
Liu's remarkable proof \cite{Liu1,Liu2} of Witten's formula about intersection numbers of moduli
spaces of principal bundles on a compact Riemann surface. Weyl's invariant theory also has important applications in
Atiyah-Singer index theory (cf. \cite{Gil}) and in Fefferman's program \cite{Fef}.

On the other hand, a natural question is: Given a polynomial $P$ in the variables
$\{g_{ij\,\alpha}\}$, find a criterion solely in terms of the coefficients of $P$
to determine whether $P$ is invariant. This problem is still vague in general.
In this paper, we give a somewhat satisfactory solution for K\"ahler manifolds.

Let $(M,g)$ be a K\"ahler manifold of dimension $n$ with K\"ahler form
$$\omega_g=\frac{\sqrt{-1}}{2\pi}\sum_{i,j=1}^n
g_{i\overline{j}}dz_i\wedge dz_{\overline{j}}.$$
Thanks to the K\"ahler condition, we can canonically associate a polynomial in the variables $\{g_{i\bar j\,\alpha}\}_{|\alpha|\geq1}$ to a semistable digraph $H$,
such that each vertex represents a partial derivative of $g_{i\bar j}$ and each edge represents the contraction of a pair of indices.
Let $\sum_{H}c(H)H$ be a linear combination of semistable graphs. In Corollary \ref{t4}, we shall give a simple criterion
to characterize its invariantness. Similar criterion will be proved for covariant differential operators in Corollary \ref{t10}.

Recall that a covariant differential operator $T^{\beta_1\cdots\beta_p}f_{/\beta_1\cdots\beta_p}$ is constructed through contractions
of curvature tensors,
$$T^{\beta_1\cdots\beta_p}=g^{\ast\ast}\cdots g^{\ast\ast}R_{\ast\ast\ast\ast/\ast\cdots\ast}\cdots R_{\ast\ast\ast\ast/\ast\cdots\ast}.$$
Their linear combinations obviously form an algebra $\mathscr R$ under the Leibnitz rule of covariant derivatives. However, we do not have a canonical basis in terms of
covariant derivatives of curvature tensors, due to additional relations like Bianchi identities and Ricci formulae. On the other hand,
the polynomials in $\{g_{i\bar j\,\alpha}\}$ associated to stable graphs form a canonical basis of $\mathscr R$. In Theorem \ref{comp}, we will prove an explicit composition formula
under this basis. On a bounded symmetric domain $\Omega$ of rank $r$, Engli\v s \cite{Eng} proved that $\mathscr R$ is equal to the algebra $\mathscr D(\Omega)$ of invariant differential operators. It is well known that $\mathscr D(\Omega)$ is a commutative algebra freely generated by $r$ algebraically independent elements. It is
an interesting and important problem to construct a set of generators explicitly (see \cite{AO, Shi, UU, Yan, Zha}).
In particular, Engli\v s \cite{Eng0} gave a set of generators using coefficients in the asymptotic expansion of the Berezin transform. From \cite{Xu2}, Engli\v s' generators can be expressed in terms of a summation over graphs (cf. Equation
\eqref{eqQ2}).

The motivation of this paper comes from deformation quantization.
Deformation quantization on a symplectic manifold $M$ was introduced in the pioneering work of Bayen et al. \cite{BFF} as a deformation of the usual
pointwise product of $C^{\infty}(M)$ into a noncommutative associative
$\star$-product of the formal series $C^\infty(M)[[\nu]]$. The celebrated work of
Kontsevich \cite{Kon,Kon2} completely solved existence and classification of star-products up to gauge equivalence
on general Poisson manifolds. Kontsevich's quantization formula was written as a summation over labeled directed graphs
with two distinguished vertices and the coefficients are certain
integrals over configuration spaces. Comprehensive surveys of deformation quantization can be found in \cite{DS} for
Poisson manifolds, and \cite{Sch2} for K\"ahler manifolds.

Let us restrict to K\"ahler manifolds $(M,g)$. A (differentiable) star product is an associative $\mathbb
\mathbb C[[\nu]]$-bilinear product $\star$ such that $\forall
f_1,f_2\in C^\infty(M)$,
\begin{equation} \label{eqberdef}
f_1\star f_2=\sum_{j=0}^\infty \nu^j C_j(f_1,f_2),
\end{equation}
where the $\mathbb C$-bilinear bidifferential operators $C_j$ satisfy
\begin{equation}\label{eqone}
C_0(f_1,f_2)=f_1 f_2,\qquad C_1(f_1,f_2)-C_1(f_2,f_1)=i\{f_1,f_2\},
\end{equation}
with the Poisson bracket $\{f_1,f_2\}$ given by
\begin{equation}\label{eqpoi}
\{f_1,f_2\}=i g^{k\bar l}\left(\frac{\partial f_1}{\partial
z^k}\frac{\partial f_2}{\partial\bar z^l}-\frac{\partial
f_2}{\partial z^k}\frac{\partial f_1}{\partial\bar z^l}\right).
\end{equation}
According to \cite{Kar,BW}, a star product has the property of \emph{separation of variables} (\emph{Wick type}), if it
satisfies $f\star h=f\cdot h$ and $h\star g=h\cdot g$ for any
locally defined antiholomorphic function $f$, holomorpihc function
$g$ and an arbitrary function $h$. If the role of holomorphic and antiholomorphic variables are swapped, we call it a star
product of \emph{anti-Wick type}.

There are earlier constructions of $\star$-products on restricted types of
K\"ahler manifolds in \cite{Ber,CGR,MO}.
Karabegov \cite{Kar} solved the classification of
deformation quantizations with separation of variables for
K\"ahler manifolds. Schlichenmaier \cite{Sch} showed that
the Berezin-Toeplitz quantization gives rise to a star product,
which turns out to be a very important quantization
with many applications. See e.g. \cite{And,BMS,Cha,Eng2,Kar2,MM,Zel}.

Feynman diagrams or directed graphs are effective tools in the construction and calculation of
star products on K\"ahler manifolds. See \cite{Gam,RT,Kar3,Xu2,Xu3}.
Inspired by work of Reshetikhin and Takhtajan \cite{RT}, Gammelgaard \cite{Gam} obtained a remarkable universal formula
in terms of acyclic graphs for a
star product with separation of variables once a
classifying Karabegov form is given.
Gammelgaard's formula crucially relies on
one's ability of writing down explicit Karabegov forms, a prototypical example is Karabegov-Schlichenmaier's identification theorem \cite{KS}.
In \cite{Xu,Xu2}, we obtained an explicit formula of Berezin star product in terms of strongly connected graphs,
which was used to give a proof of an explicit formula of Berezin-Toeplitz star product
due to Gammelgaard, Karabegov and Schlichenmaier. Karabegov \cite{Kar3} recently gave a very insightful algebraic proof of Gammelgaard's formula.

We will prove in Theorem \ref{thmstar} explicit formulae of
star products whose Karabegov forms are
summations over strongly connected graphs.

\

\noindent{\bf Acknowledgements}
The author thanks Professors Kefeng Liu and Shing-Tung Yau for helpful conversations and encouragements.
The author also thanks Professors Alexander Karabegov and Martin Schlichenmaier for stimulating communications
on deformation quantization.

\vskip 30pt
\section{Covariant tensors in semistable trees} \label{secpd}

Throughout this paper, a \emph{digraph} or simply a graph $G=(V,E)$ is defined to be a finite directed multigraph which is permitted to have multi-edges and
loops.

A vertex $v$ of a digraph $G$ is called \emph{stable} if $\deg^-(v)\geq2,\ \deg^+(v)\geq2$,
i.e. both the inward and outward degrees of $v$ are no less than $2$. A vertex $v$ is called \emph{semistable} if we have
$$\deg^-(v)\geq1,\ \deg^+(v)\geq1,\ \deg^-(v)+\deg^+(v)\geq3.$$

\begin{definition}\label{dfdecotree}
A \emph{decorated tree} $T$ is a directed tree that each vertex is decorated by a finite number of outward and inward
external legs, corresponding to unbarred and barred indices respectively. $T$ is called \emph{semistable} (resp. \emph{stable}) if each vertex is semistable
(resp. stable).
The inward (resp. outward) degree of a vertex $v$ is defined to be the number of inward (resp. outward) half-edges at $v$.
Note that a half-edge may refer to
the head or tail of an edge of $T$ or an external leg.
\end{definition}

\begin{definition}
A directed edge $uv$ of a semistable decorated tree or a semistable digraph is called \emph{contractible} if $u\neq v$ and at least one of the following two conditions holds:
(i) $\deg^+(u)=1$; (ii) $\deg^-(v)=1$.
\end{definition}

\begin{lemma}\label{t1}
Let $T$ be a semistable decorated tree. Denote by $T'$ a tree obtained by contracting a finite number of contractible edges in $T$.
Then $T'$ is also semistable and an edge in $T'$ is contractible if and only if it is contractible in $T$.
\end{lemma}
\begin{proof}
Let $uv$ be a contractible edge of $T$. Let $T'$ be the tree obtained by contracting $uv$ and $p$ the new vertex
merging $u$ and $v$. Then obviously $\deg_{T'} p\geq4$. We also have $\deg_{T'}^- p\geq1$ and $\deg_{T'}^+ p\geq1$, since
$u$ has at least one inward half-edge and $v$ has at least one outward half-edge. So we proved that $T'$ is semistable.

Let $e$ be an edge of $T$ other than $uv$. If $e$ is not incident to $u$ or $v$, then it is obvious that $e$ is contractible in $T$
if and only if it is contractible in $T'$.

If $e=vw$ is contractible in $T$, then there are two cases: (i) $\deg_{T}^- w=\deg_{T'}^- w=1$;
(ii) $\deg_{T}^- w\neq 1$ and $\deg_T^+v=1$. In Case (i), it is obvious that $e$ is also contractible in $T'$.
In Case (ii), we have $\deg_T^-v\geq2$, so the contractibility of $uv$ in $T$ implies $\deg_T^+u=1$, namely
$\deg_{T'}^+p=1$. Thus $pw$ is contractible in $T'$.

If $e=vw$ is non-contractible in $T$, then $\deg_{T'}^- w=\deg_{T}^- w\geq2$
and $\deg_{T'}^+ p\geq\deg_{T}^+ v\geq2$. So $pw$ is non-contractible in $T'$.

The same argument works when $e=wu$. We conclude the proof.
\end{proof}

\begin{definition}
A semistable decorated tree $T$ is called \emph{contractible} if all of its edges are contractible.
Denote by $\mathscr T_g(a_1\cdots a_k | \bar b_1\cdots\bar b_m)$ the set of all contractible semistable decorated trees
with external legs in the set $\{a_1\cdots a_k, \bar b_1\cdots\bar b_m\}$. Denote by $t_{k,m}(n)$ the number of $n$-vertex trees in $\mathscr T_g(a_1\cdots a_k | \bar b_1\cdots\bar b_m)$. The first values of $t_{k,m}(n)$ were listed in Table \ref{tb1} of Appendix \ref{ap}.
\end{definition}

Denote by $D(g_{a_1\bar b_1 a_2\cdots a_k\bar b_2\cdots\bar b_m})$ the canonical invariant Weyl polynomial that equals $g_{a_1\bar b_1 a_2\cdots a_k\bar b_2\cdots\bar b_m}$ at the center of a normal coordinate system.
A proof of the following theorem was outline in \cite[\S 2]{Xu}, where a contractible tree was equivalently defined as an indecomposable admissible tree.
Here we give a more direct proof.
\begin{theorem} \label{t8} Let $k,m\geq2$. Then
\begin{equation}\label{eqt1}
D(g_{a_1\bar b_1 a_2\cdots a_k\bar b_2\cdots\bar b_m})=\sum_{T\in\mathscr T_g(a_1\cdots a_k | \bar b_1\cdots\bar b_m)}(-1)^{|V(T)|+1} g_T,
\end{equation}
where $g_T$ is the Weyl invariant associated to $T$.
\end{theorem}
\begin{proof}
Note that a contractible semistable tree with no less than two vertices must have a vertex which is not stable; therefore the right-hand side of
\eqref{eqt1} is equal to $g_{a_1\bar b_1 a_2\cdots a_k\bar b_2\cdots\bar b_m}$ at the center of a normal coordinate system.
Thus we need only prove that the right-hand side of
\eqref{eqt1} is a covariant tensor with indices $(a_1\cdots a_k, \bar b_1\cdots\bar b_m)$. Let $\phi$ be a local biholomorphic
mapping. Under the change of coordinates $x\mapsto\phi(x)$, we have
\begin{gather*}
g_{i\bar j}(x)=g_{p\bar q}(\phi(x))(\partial_i\phi_p)(\overline{\partial_j\phi_q})=\xymatrix@C=7mm@R=5mm{ \ar[r]^{\partial_{\bar j}\bar\phi} & \circ \ar[r]^{\partial_i\phi} & },\\
g_{i\bar j l}(x)=g_{p\bar q r}(\phi(x))(\partial_l\phi_r)(\partial_i\phi_p)(\overline{\partial_j\phi_q})+g_{p\bar q}(\phi(x))(\partial_{il}\phi_p)(\overline{\partial_j\phi_q})\\
=\begin{tabular}{c}
\xymatrix@C=7mm@R=5mm{
&&
\\ \ar[r]^{\partial_{\bar j}\bar\phi} &\circ \ar[r]_{\partial_l\phi} \ar[u]_{\partial_i\phi} & }
\end{tabular}
+\begin{tabular}{c}
\xymatrix@C=7mm@R=5mm{ \ar[r]^{\partial_{\bar j}\bar\phi} & \circ \ar[r]^{\partial_{il}\phi} & }
\end{tabular}.
\end{gather*}
The graphical expressions will make the proof much easier. In general, it is not difficult to see that
\begin{equation}\label{eqt2}
g_{a_1\bar b_1 a_2\cdots a_k\bar b_2\cdots\bar b_m}(x)=\sum_{P\in\text{partition}(a_1\cdots a_k | \bar b_1\cdots\bar b_m)}\circ_P,
\end{equation}
where $P$ runs over all partitions of the set $(a_1\cdots a_k, \bar b_1\cdots\bar b_m)$ such that no subset contains both barred and unbarred
indices. $\circ_P$ denotes a single vertex decorated by external legs which are in one-to-one correspondence with elements of $P$.
Trees in the right-hand side of \eqref{eqt2}
may be called $\phi$-trees.

Let us first look at how Equation \eqref{eqt2} works by showing that $g_{i\bar jp\bar q}-g^{r\bar s}g_{r\bar j\bar q}g_{i\bar sp}$ is a covariant
tensor. By \eqref{eqt2}, we have
\begin{gather*}\label{eqt4}
g_{i\bar jp\bar q}(x)=\begin{tabular}{c}
\xymatrix@C=7mm@R=5mm{
&&
\\ \ar[r]^{\partial_{\bar q}\bar\phi} &\circ \ar[r]_{\partial_k\phi} \ar[u]_{\partial_i\phi} &
\\ \ar[ur]_{\partial_{\bar j}\bar\phi}&&}
\end{tabular}
+\begin{tabular}{c}
\xymatrix@C=7mm@R=5mm{
&\ar[d]^{\partial_{\bar j}\bar\phi}&
\\ \ar[r]^{\partial_{\bar q}\bar\phi} &\circ \ar[r]_{\partial_{ip}\phi} & }
\end{tabular}+\begin{tabular}{c}
\xymatrix@C=7mm@R=5mm{
&&
\\ \ar[r]^{\partial_{\bar j\bar q}\bar\phi} &\circ \ar[r]_{\partial_p\phi} \ar[u]_{\partial_i\phi} & }
\end{tabular}
+\begin{tabular}{c}
\xymatrix@C=7mm@R=5mm{ \ar[r]^{\partial_{\bar j\bar q}\bar\phi} & \circ \ar[r]^{\partial_{ip}\phi} & }
\end{tabular}\\
g_{r\bar j\bar q}(x)=\begin{tabular}{c}
\xymatrix@C=7mm@R=5mm{
&\ar[d]^{\partial_{\bar j}\bar\phi}&
\\ \ar[r]^{\partial_{\bar q}\bar\phi} &\circ \ar[r]_{\partial_{r}\phi} & }
\end{tabular}
+\begin{tabular}{c}
\xymatrix@C=7mm@R=5mm{ \ar[r]^{\partial_{\bar j\bar q}\bar\phi} & \circ \ar[r]^{\partial_r\phi} & }
\end{tabular}\\
g_{i\bar sp}(x)=\begin{tabular}{c}
\xymatrix@C=7mm@R=5mm{
&&
\\ \ar[r]^{\partial_{\bar s}\bar\phi} &\circ \ar[r]_{\partial_p\phi} \ar[u]_{\partial_i\phi} & }
\end{tabular}
+\begin{tabular}{c}
\xymatrix@C=7mm@R=5mm{ \ar[r]^{\partial_{\bar s}\bar\phi} & \circ \ar[r]^{\partial_{ip}\phi} & }
\end{tabular}
\end{gather*}
By $g^{r\bar s}(x)=g^{c\bar d}(\phi(x))(\partial_c\phi_r^{-1})(\overline{\partial_d\phi_s^{-1}})$, we have
\begin{multline}\label{eqt3}
g^{r\bar s}(x)g_{r\bar j\bar q}(x)g_{i\bar sp}(x)=
\begin{tabular}{c}
\xymatrix@C=6mm@R=4mm{
\ar[dr]^{\partial_{\bar j}\bar\phi} & & &
\\ &\circ \ar[r]  & \circ
\ar[ur]^{\partial_i\phi} \ar[dr]_{\partial_p\phi} &
\\ \ar[ur]_{\partial_{\bar q}\bar\phi} &&&}
\end{tabular}
+\begin{tabular}{c}
\xymatrix@C=7mm@R=5mm{
&\ar[d]^{\partial_{\bar j}\bar\phi}& &
\\ \ar[r]^{\partial_{\bar q}\bar\phi} &\circ\ar[r] &\circ \ar[r]_{\partial_{ip}\phi} & }
\end{tabular}\\
+\begin{tabular}{c}
\xymatrix@C=7mm@R=5mm{
&&&
\\ \ar[r]^{\partial_{\bar j\bar q}\bar\phi}&\circ\ar[r] &\circ \ar[r]_{\partial_p\phi} \ar[u]_{\partial_i\phi} & }
\end{tabular}
+\begin{tabular}{c}
\xymatrix@C=7mm@R=5mm{ \ar[r]^{\partial_{\bar j\bar q}\bar\phi}&\circ\ar[r] & \circ \ar[r]^{\partial_{ip}\phi} & }
\end{tabular},
\end{multline}
where we used $\sum_r(\partial_c\phi_r^{-1})(\partial_r\phi_t)=\delta_{ct}$. For the same reason, an internal edge $e=uv$ of a $\phi$-tree
could be contracted if both half-edges of $e$ has no decoration and either $\deg^+v=1$ or $\deg^-u=1$. Therefore the unique edge in the last three trees at the
right-hand side of \eqref{eqt3} could be contracted. The resulting $\phi$-trees cancel with the corresponding $\phi$-trees from $g_{i\bar jp\bar q}(x)$.
Thus we get
$$
g_{i\bar jp\bar q}(x)-g^{r\bar s}(x)g_{r\bar j\bar q}(x)g_{i\bar sp}(x)=
\begin{tabular}{c}
\xymatrix@C=7mm@R=5mm{
&&
\\ \ar[r]^{\partial_{\bar q}\bar\phi} &\circ \ar[r]_{\partial_k\phi} \ar[u]_{\partial_i\phi} &
\\ \ar[ur]_{\partial_{\bar j}\bar\phi}&&}
\end{tabular}
-\begin{tabular}{c}
\xymatrix@C=6mm@R=4mm{
\ar[dr]^{\partial_{\bar j}\bar\phi} & & &
\\ &\circ \ar[r]  & \circ
\ar[ur]^{\partial_i\phi} \ar[dr]_{\partial_p\phi} &
\\ \ar[ur]_{\partial_{\bar q}\bar\phi} &&&}
\end{tabular},
$$
which implies that $g_{i\bar jp\bar q}-g^{r\bar s}g_{r\bar j\bar q}g_{i\bar sp}$ is a covariant
tensor.

From the above discussion, we see that under the change of coordinates $x\mapsto\phi(x)$, the right-hand side of \eqref{eqt1}
\begin{equation}\label{eqt5}
\sum_{T\in\mathscr T_g(a_1\cdots a_k | \bar b_1\cdots\bar b_m)}(-1)^{|V(T)|+1} g_T(x)
\end{equation}
is equal to a summation of $\phi$-trees whose (internal or external) half-edges are decorated by indices $(\partial_{a_1}\phi\cdots \partial_{a_k}\phi, \partial_{\bar b_1}\bar\phi\cdots\partial_{\bar b_m}\bar\phi)$. In order to prove \eqref{eqt1}, it is enough to prove that for any ill-decorated $\phi$-tree $T_\phi$ (i.e. some internal half-edge is decorated or some
external leg with multiple derivatives), then its coefficient is zero in the above summation.
We need to enumerate all trees in the summation \eqref{eqt5} that may produce $T_\phi$.

Again it is illuminating to look at an example first. Consider the following two ill-decorated $\phi$-trees.
\begin{equation}\label{eqt6}
\begin{tabular}{c}
\xymatrix@C=7mm@R=4mm{
\ar[dr]^{\partial_{\bar j}\bar\phi} & & &
\\ &\circ \ar^(.3){\partial_r\phi}[r]  & \circ
\ar[ur]^{\partial_i\phi} \ar[dr]_{\partial_p\phi} &
\\ \ar[ur]_{\partial_{\bar q}\bar\phi} &&&}
\end{tabular}\qquad
\begin{tabular}{c}
\xymatrix@C=7mm@R=4mm{
\ar[dr]^{\partial_{\bar j}\bar\phi} & & &
\\ &\circ \ar[r]  & \circ
\ar[ur]^{\partial_{ip}\phi} \ar[dr]_{\partial_r\phi} &
\\ \ar[ur]_{\partial_{\bar q}\bar\phi} &&&}
\end{tabular}
\end{equation}
For the first $\phi$-tree of \eqref{eqt6}, the left vertex is ill-decorated. It may come from two contractible semistable trees with opposite signs.
$$\begin{tabular}{c}
\xymatrix@C=7mm@R=4mm{
\ar[dr]^{\bar j} &&  &
\\ &\circ \ar[r] \ar[d]^{r} & \circ
\ar[ur]^{i} \ar[dr]^{p} &
\\ \ar[ur]^{\bar q} &&&}\qquad
\xymatrix@C=5mm@R=4mm{
\ar[dr]^{\bar j} && & &
\\ &\circ \ar[r]&\circ\ar[r] \ar[d]^{r}& \circ
\ar[ur]^{i} \ar[dr]^{p} &
\\ \ar[ur]^{\bar q} &&&&}
\end{tabular}$$
For the second $\phi$-tree of \eqref{eqt6}, the right vertex is ill-decorated. It may come from two contractible semistable trees with opposite signs.
$$\begin{tabular}{c}
\xymatrix@C=7mm@R=4mm{
\ar[dr]^{\bar j} &&  &
\\ &\circ \ar[r]  & \circ
\ar[u]^{i}\ar[ur]^{p} \ar[dr]^{r} &
\\ \ar[ur]^{\bar q} &&&}\qquad
\xymatrix@C=5mm@R=4mm{
\ar[dr]^{\bar j} && & &
\\ &\circ \ar[r]&\circ\ar[r] \ar[d]^{r}& \circ
\ar[ur]^{i} \ar[dr]^{p} &
\\ \ar[ur]^{\bar q} &&&&}
\end{tabular}$$
The above process may be called ``freeing ill-decorated indices''.

For a general $\phi$-tree, we may treat each ill-decorated vertex separately. If the vertex has degree $2$,
then we have the following two ways to free ill-decorated indices, their numbers of vertices differ by $1$.
\begin{equation*}
\begin{tabular}{c}
\xymatrix@C=7mm@R=5mm{\cdots \ar[r]^{\partial_{\bar *}\bar\phi} & \circ \ar[r]^{\partial_{\bullet}\phi} & \cdots}
\end{tabular}\Longrightarrow
\begin{tabular}{c}
\xymatrix@C=7mm@R=5mm{
&&
\\ \cdots\ar[r] &\circ \ar[r] \ar[u]_{\bullet} &\cdots
\\&\ar[u]_{\bar *}&}
\end{tabular}
\begin{tabular}{c}
\xymatrix@C=7mm@R=5mm{\ar@<1.4ex>[d]^{\bar *}&\\
\cdots\circ \ar[r] & \circ\cdots\ar@<1.6ex>[u]_{\bullet}}
\end{tabular}
\end{equation*}
Namely the ill-decorated inward (resp. outward) indices may be separated and attached to either the original vertex or to a new vertex at the tail (resp. head)
of the half-edge. It is obvious that the new edge is contractible.

If a vertex $v$ has degree no less than $3$, there are $2^{c(v)}$ ways of freeing ill-decorated indices, where $c(v)$ is the
total number of decorated internal half-edges and
external legs with multiple derivatives incident to $v$. It is easy to see that they add up to zero. So we conclude the proof.
\end{proof}

As an example, we can compute directly that
\begin{multline*}
D(g_{i\bar j k\bar l p})=-R_{i\bar j k\bar l p}=-\partial_p R_{i\bar j k\bar l}+\Gamma_{pi}^\delta R_{\delta\bar j k\bar l}+\Gamma_{pk}^\delta R_{i\bar j\delta\bar l}\\
=\begin{tabular}{c}
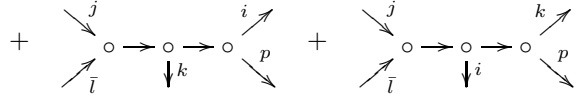
\xymatrix@C=4mm@R=3mm{ \ar[dr]^{\bar j} & &\\  &\circ \ar[ur]^i\ar[r]
\ar[dr]_p & {\scriptstyle k}
\\  \ar[ur]_{\bar l} & &}
\end{tabular}
-\begin{tabular}{c}
\xymatrix@C=4mm@R=3mm{
\ar[dr]^{\bar j} & & &
\\ &\circ \ar[r]  & \circ
\ar[ur]^{i}\ar[r] \ar[dr]_{p} & {\scriptstyle k}
\\ \ar[ur]_{\bar l} &&&}
\end{tabular}
-\begin{tabular}{c}
\xymatrix@C=4mm@R=3mm{
\ar[dr]^{\bar j} &&  &
\\ &\circ \ar[r] \ar[d]^{p} & \circ
\ar[ur]^{i} \ar[dr]^{k} &
\\ \ar[ur]^{\bar l} &&&}
\end{tabular}\\
-\begin{tabular}{c}
\xymatrix@C=4mm@R=3mm{
\ar[dr]^{\bar j} &&  &
\\ &\circ \ar[r] \ar[d]^{k} & \circ
\ar[ur]^{i} \ar[dr]^{p} &
\\ \ar[ur]^{\bar l} &&&}
\end{tabular}
-\begin{tabular}{c}
\xymatrix@C=4mm@R=3mm{
\ar[dr]^{\bar j} &&  &
\\ &\circ \ar[r] \ar[d]^{i} & \circ
\ar[ur]^{k} \ar[dr]^{p} &
\\ \ar[ur]^{\bar l} &&&}
\end{tabular}
+\begin{tabular}{c}
\xymatrix@C=4mm@R=3mm{
\ar[dr]^{\bar j} && & &
\\ &\circ \ar[r]&\circ\ar[r] \ar[d]^{p}& \circ
\ar[ur]^{i} \ar[dr]^{k} &
\\ \ar[ur]_{\bar l} &&&&}
\end{tabular}\\
+\begin{tabular}{c}
\xymatrix@C=4mm@R=3mm{
\ar[dr]^{\bar j} && & &
\\ &\circ \ar[r]&\circ\ar[r] \ar[d]^{k}& \circ
\ar[ur]^{i} \ar[dr]^{p} &
\\ \ar[ur]_{\bar l} &&&&}
\end{tabular}
+\begin{tabular}{c}
\xymatrix@C=4mm@R=3mm{
\ar[dr]^{\bar j} && & &
\\ &\circ \ar[r]&\circ\ar[r] \ar[d]^{i}& \circ
\ar[ur]^{k} \ar[dr]^{p} &
\\ \ar[ur]_{\bar l} &&&&}
\end{tabular}
\end{multline*}
which agrees with \eqref{eqt1}.

\vskip 30pt
\section{Weyl invariants in semistable graphs} \label{secsemi}

The weight of a digraph $G$ is defined to be the integer $w(G)=|E|-|V|$.
A digraph
$G$ is \emph{stable} (resp. \emph{semistable}) if each vertex of $G$ is stable (resp. semistable).
The set of semistable and stable graphs of weight $k$ will
be denoted by $\mathcal G^{ss}(k)$ and $\mathcal G(k)$ respectively.

A digraph $G$ is called \emph{strongly connected} or \emph{strong} if there is a
directed path from each vertex in $G$ to every other vertex.
The strongly connected components (SCC) of a digraph $G$ can each be contracted to a single vertex,
the resulting graph is a directed acyclic graph (DAG), called the \emph{condensation} of $G$.
A \emph{source} (resp. \emph{sink}) of $G$ is a SCC that has only outward (resp. inward) edges in the condensation of $G$.

\begin{lemma}\label{t2}
Let $e=uv$ be a contractible edge of a semistable graph $G$. Denote by $G'$ the graph obtained by contracting $e$ in $G$.
If $e'\neq e$ is an edge of $G$ such that $e'\neq vu$,
then $e'$ is contractible in $G$ if and only if it is contractible in $G'$.
\end{lemma}
\begin{proof}
The proof is almost identical to the proof of Lemma \ref{t1}.
\end{proof}

\begin{definition}
A semistable graph $G$ is called \emph{stabilizable} if after contractions of a finite number of contractible edges of
$G$, the resulting graph becomes stable, which is called the \emph{stabilization graph} of $G$ and denoted by $G^s$.
\end{definition}

If $G$ is the stabilization graph of $H$, then $w(G)=w(H)$.
By Lemma \ref{t2}, the stablizability of a semistable graph $G$ is independent of the order of edge-contractions.

\begin{lemma}\label{t5}
A strong semistable graph $G$ is stabilizable.
\end{lemma}
\begin{proof}
Let $v$ be a nonstable vertex of $G$. Then $v$ has no loop by the strongness of $G$. Moreover, $v$ has either a
unique inward or a unique outward edge, which is contractible. Thus we can always contract some edge
until a stable graph is reached.
\end{proof}

A connected semistable graph may be not stabilizable, e.g.
$\xymatrix{*+[o][F-]{1} \ar[r]^{1} &
*+[o][F-]{1}}$.

\begin{lemma}\label{t6}
Let $G$ be a stabilizable semistable graph. If the stabilization graph of $G$ is strong, then $G$ is also strong.
\end{lemma}
\begin{proof}
Obviously $G$ is connected. If $G$ is not strong, first assume that $G$ has two SCC's $A,B$. Then it is not difficult to see
 that any edge between $A,B$ is not contractible, a contradiction. If $G$ has more than two SCC's, consider its condensation $G'$.
Choose any edge $e$ in $G'$ which is contractible in $G$, we can contract $e$ to reduce the number SCC's of $G$ by one.
Since the stabilization graph of $G$ is strong, we can always repeat this process until we get a graph with two SCC's, which is not contractible,
a contradiction again. Therefore $G$ must be strong.
\end{proof}

\begin{theorem}\label{t3}
Let $G$ be a stable graph of weight $k$. Then
\begin{equation}\label{eqt7}
D(G)=\sum_{H\in \mathcal G^{ss}(k)}^{\text{stabilizable}}\frac{(-1)^{|V(H)|-|V(G)|}|{\rm Aut(G)}|}{|{\rm Aut(H)}|}\, H,
\end{equation}
where $H$ runs over stabilizable semistable graphs of weight $k$ whose stabilization graph is $G$.
\end{theorem}
\begin{proof}
By definition, $D(G)$ is a sum of stabilizable semistable graphs obtained by expanding each vertex of $G$ by \eqref{eqt1} as a sum of contractible semistable trees, while
keeping incidence relations of $G$.
The group ${\rm Aut}(G)$
has a natural action on the above multiset
of stabilizable semistable graphs $H$ in the expansion of $D(G)$. Then it is not
difficult to see that the set of orbits is in one-to-one correspondence with isomorphism
classes of stabilizable semistable graphs of weight $k$ and the isotropy group at $H$ is ${\rm Aut}(H)$. Therefore
the orbit of $H$ has $|{\rm Aut(G)}|/|{\rm Aut(H)}|$ graphs. The factor $(-1)^{|V(H)|-|V(G)|}$ is clear from \eqref{eqt1}. So we conclude the
proof of \eqref{eqt7}.
\end{proof}

\begin{corollary}\label{t4}
A linear combination of stabilizable semistable graphs of weight $k$
\begin{equation}\label{eqt8}
\sum_{H\in \mathcal G^{ss}(k)}^{\text{stabilizable}}c(H)\frac{(-1)^{|V(H)|}}{|{\rm Aut(H)}|}H
\end{equation}
is a Weyl invariant (i.e. invariant under coordinate transformations) if and only if
$c(H_1)=c(H_2)$ whenever $H_1,H_2$ have the same stabilization graph.
\end{corollary}
\begin{proof}
Note that \eqref{eqt8} is a Weyl invariant if and only if it is equal to
$$\sum_{G\in \mathcal G(k)}c(G)\frac{(-1)^{|V(G)|}}{|{\rm Aut(G)}|}D(G).$$
So the corollary follows from Theorem \ref{t3}.
\end{proof}

\begin{definition}
For convenience, a function $c(H)$ defined on the set of stabilizable semistable graphs is called a \emph{Weyl function}
if it satisfies $c(H_1)=c(H_2)$ whenever $H_1,H_2$ have the same stabilization graph.
\end{definition}

Any constant function is a Weyl function. Below is a more nontrivial example.
\begin{lemma} \label{t11}
Let $\mathscr L(H)$ be the set of linear subgraphs of $H$ (note $\emptyset\in\mathscr L(H)$) and $p(L)$ the number of components of $L\in\mathscr L(H)$.
Then
\begin{equation}\label{eqC}
\beta_C(H)=\sum_{L\in\mathscr L(H)} C^{p(L)},
\end{equation}
is a Weyl function for any constant $C$.
\end{lemma}
\begin{proof}
Let $H'$ be a graph obtained by contracting a contractible edge $e=uv$ in $H$.
For any given $L\in\mathscr L(H)$, define $L'\in\mathscr L(H')$ by
$$L'=\begin{cases} L, & e\notin L
\\
L/\{e\}, & e\in L
\end{cases}$$
where $L/\{e\}$ is the graph obtained by contracting $e$ in $L$.
Since we have either $\deg^+ u=0$ or $\deg^- v=0$,
it is not difficult to see that $L\mapsto L'$ gives a one-to-one correspondence between $\mathscr L(H)$ and $\mathscr L(H')$. Moreover, $p(L)=p(L')$.
So we have
\begin{equation*}
\beta_C(H)=\sum_{L\in\mathscr L(H)} C^{p(L)}=\sum_{L'\in\mathscr L(H')} C^{p(L')}=\beta_C(H').
\end{equation*}
This implies $\beta_C(H_1)=\beta_C(H_2)$ whenever $H_1,H_2$ have the same stabilization graph.
\end{proof}

\begin{corollary} \label{t12}
(i) $\det(I-A(H))$ is a Weyl function, where $I$ is the identity matrix and $A(H)$ is the adjacency matrix of $H$.

(ii) $|\mathscr L(H)|$, the number of linear subgraphs of $H$, is a Weyl function.
\end{corollary}
\begin{proof}
(i) follows by taking $C=-1$ in \eqref{eqC} and using the following Coefficient Theorem from spectral graph theory,
\begin{equation}\label{eqc}
\det(I-A(H))=\sum_{L\in\mathscr L(H)} (-1)^{p(L)}.
\end{equation}

(ii) follows by taking $C=1$ in \eqref{eqC}.
\end{proof}

We remark that $\det(I-A(H))$ appears as the coefficients of asymptotic expansions of the Bergman kernel \cite{Xu}.

\vskip 30pt
\section{Covariant differential operators} \label{secdiff}

Differential operators on K\"ahler manifolds can be encoded by digraphs with a distinguished vertex.
The results in previous sections can be extended to this setting almost verbatim.

\begin{definition}
A \emph{(one-)pointed tree} $T=(V\cup\{\bullet\})$ is
defined to be a decorated tree with a distinguished vertex labeled by $f$. A \emph{(one-)pointed graph} $\Gamma=(V\cup\{\bullet\},E)$ is
defined to be a digraph with a distinguished vertex labeled by $f$.
$T$ or $\Gamma$ is called \emph{semistable} (resp. \emph{stable}) if each ordinary vertex $v\in V$ is semistable
(resp. stable).
\end{definition}

\begin{definition}
A directed edge $uv$ of a semistable pointed tree or a semistable pointed graph is called \emph{contractible} if $u\neq v$ and at least one of the following two conditions holds:
(i) $u\in V$ and $\deg^+(u)=1$; (ii) $v\in V$ and $\deg^-(v)=1$.
\end{definition}

A semistable pointed tree $T$ is called \emph{contractible} if all of its edges are contractible. Note that
Lemma \ref{t1} still holds for pointed trees.

\begin{theorem} Let $k,m\geq0$. Then
\begin{equation}\label{eqt11}
D(f_{a_1\cdots a_k\bar b_1\cdots\bar b_m})=\sum_{T=(V\cup\{\bullet\})\in\mathscr T_f(a_1\cdots a_k | \bar b_1\cdots\bar b_m)}(-1)^{|V|} f_T,
\end{equation}
where $\mathscr T_f(a_1\cdots a_k | \bar b_1\cdots\bar b_m)$ the set of all contractible semistable pointed trees
with external legs in the set $\{a_1\cdots a_k, \bar b_1\cdots\bar b_m\}$ and $f_T$ is the Weyl invariant associated to the pointed tree $T$.
\end{theorem}
\begin{proof}
The proof is similar to Theorem \ref{t8}.
\end{proof}

\begin{definition}
The weight of a pointed graph $\Gamma=(V\cup\{\bullet\},E)$ is defined to be $w(\Gamma)=|E|-|V|$.
By abuse of notation, we denote $V(\Gamma)=V\cup\{\bullet\}$.
The set of semistable and stable pointed graphs of weight $k$ will
be denoted by $\mathcal G^{ss}_1(k)$ and $\mathcal G_1(k)$ respectively.
 We denote
by ${\rm Aut}(\Gamma)$ the set of all automorphisms of the pointed graph
$\Gamma$ fixing the distinguished vertex.
\end{definition}

A semistable pointed graph $\Gamma$ is called \emph{stabilizable} if after contractions of a finite number of contractible edges of
$\Gamma$, the resulting graph becomes stable, which is called the \emph{stabilization graph} of $\Gamma$ and denoted by $\Gamma^s$. Note that
Lemma \ref{t2} still holds for pointed graphs.

\begin{lemma}\label{t7}
(i) A strong semistable pointed graph $\Gamma$ is stabilizable.

(ii) Let $\Gamma$ be a stabilizable semistable graph. If the stabilization graph of $\Gamma$ is strong, then $\Gamma$ is also strong.
\end{lemma}
\begin{proof}
The proof is similar to Lemma \ref{t5} and Lemma \ref{t6}.
\end{proof}

\begin{theorem}\label{t9}
Let $\Gamma$ be a stable pointed graph of weight $k$. Then
\begin{equation}\label{eqt12}
D(\Gamma)=\sum_{Z}\frac{(-1)^{|V(Z)|-|V(\Gamma)|}|{\rm Aut(\Gamma)}|}{|{\rm Aut(Z)}|}  Z,
\end{equation}
where $Z$ runs over stabilizable semistable pointed graphs of weight $k$ whose stabilization graph is $\Gamma$.
\end{theorem}
\begin{proof}
The proof is similar to Theorem \ref{t3}.
\end{proof}

\begin{corollary}\label{t10}
A linear combination of stabilizable semistable pointed graphs
\begin{equation}\label{eqt13}
\sum_{Z\in \mathcal G^{ss}_1}^{\text{stabilizable}}c(Z)\frac{(-1)^{|V(Z)|}}{|{\rm Aut(Z)}|}Z
\end{equation}
is a covariant differential operator (i.e. invariant under coordinate transformations) if and only if
$c(Z_1)=c(Z_2)$ whenever $Z_1,Z_2$ have the same stabilization graph.
\end{corollary}
\begin{proof}
It follows immediately from Theorem \ref{t9}.
\end{proof}

\begin{example}
Engli\v s \cite{Eng} proved the following asymptotic expansion for a Laplace integral
on a domain $\Omega\in\mathbb C^n$ when $ m\rightarrow\infty$,
\begin{equation}
\int_{\Omega} f(y)e^{-m(\Phi(x, x) + \Phi(y, y)-\Phi(x, y)-\Phi(y, x))}\frac{\omega_g^n(y)}{n!} =
\frac{1}{m^n}\sum_{k\geq0}m^{-k}R_k(f)(x),
\end{equation}
where $\Phi$ is the K\"ahler potential and $R_k$ are
covariant differential operators.

In \cite[Thm 3.2]{Xu2}, we proved an explicit formula for $R_k$,
\begin{equation}\label{eqt22}
R_k(f)=\sum_{\Gamma\in \mathcal G^{ss}_1} \frac{\det(A(\Gamma_-)-I)}{|{\rm Aut}(\Gamma)|}\,\Gamma,
\end{equation} 
where $\Gamma_-$ is obtained by removing
the distinguished vertex of $\Gamma$. 

We show that \eqref{eqt22} is consistent with Corollary \ref{t10}. Similar to Corollary \ref{t12} (i), we have
$\det(I-A(\Gamma_-))=\det(I-A(\Gamma'_-))$ where $\Gamma'$ is obtained by contracting a contractible edge in $\Gamma$.
Moreover, if $\Gamma$ is a semistable pointed graph which is non-stabilizable, then $\det(I-A(\Gamma_-))=0$. In order to prove the last assertion,
we may assume that each edge of $\Gamma$ is non-contractible. If $v$ is a strictly semistable ordinary vertex (i.e. $\deg(v)=3$), then $v$ must have
a self-loop, namely $\Gamma_-$ contains a SCC $\{\xymatrix{*+[o][F-]{1}}\}$. Therefore we must have $\det(I-A(\Gamma_-))=0$.
\end{example}

By Corollary \ref{t10}, the graded algebra $\mathscr R$ of abstract covariant differential operators has a canonical basis $\mathcal G_1$ consisting of
stable pointed graphs, graded by weights.
Before we give a multiplication formula in this algebra, we need more definitions.
\begin{definition} \label{dfGS}
Let $\Gamma=(V\cup\{\bullet\},E)$ be a pointed graph that can be
obtained by inserting a finite number of vertices to edges of a semistable pointed graph $\Gamma^{ss}$, called the
semistabilization graph of $\Gamma$. Such $\Gamma$ is called \emph{generalized stabilizable} (GS) if
$\Gamma^{ss}$ is stabilizable. The stabilization graph of $\Gamma^{ss}$, denoted by $\Gamma^s$, is also called the stabilization graph of $\Gamma$.
\end{definition}

The reason we introduce GS pointed graphs is to account for the derivatives on edges $g^{i\bar j}$.
See \cite[Rem. 3.7]{Xu2} for detailed discussions.

We have the following explicit composition formula of covariant differential operators.
\begin{theorem}\label{comp} In terms of the basis of stable pointed graphs, we have
\begin{multline}\label{eqt14}
\left(\sum_{Z_1\in \mathcal G_1}c_1(Z_1)\frac{(-1)^{|V(Z_1)|}}{|{\rm Aut(Z_1)}|}Z_1\right) \circ
\left(\sum_{Z_2\in \mathcal G_1}c_2(Z_2)\frac{(-1)^{|V(Z_2)|}}{|{\rm Aut(Z_2)}|}Z_2\right)\\
=\sum_{Z\in \mathcal G_1}\left(\sum_{\Gamma\subset Z}^{\text{GS}}(-1)^{|V((Z/\Gamma)^s)|+|V(\Gamma^s)|}
c_1((Z/\Gamma)^s)c_2(\Gamma^s)\right)\frac{1}{|{\rm Aut(Z)}|}Z,
\end{multline}
where $\Gamma$ runs over all GS pointed subgraph of $Z$, and $Z/\Gamma$ is the pointed graph
obtained from $Z$ by contracting $\Gamma$ to a point.
\end{theorem}
\begin{proof}
Equation \eqref{eqt14} follows almost immediately from
Corollary \ref{t10} and results proved in our previous paper \cite[Lem. 3.10 \& Rem. 3.7]{Xu3}.
\end{proof}

A further justification to \eqref{eqt14} is
the following lemma.
\begin{lemma}
Let $Z$ be a GS pointed graph and $\Gamma$ a GS pointed subgraph of $Z$. Then $Z/\Gamma$ is also a GS pointed graph.
\end{lemma}
\begin{proof}
Let $e$ be an edge in $(Z/\Gamma)^{ss}$. Then it is not difficult to see that $e$ is contractible in $(Z/\Gamma)^{ss}$ if and only if
$e$ is contractible in $Z^{ss}$. Therefore $(Z/\Gamma)^{ss}$ is stabilizable, since $Z^{ss}$ is stabilizable.
\end{proof}

By Lemma \ref{t7}, each strong pointed graph is a GS pointed graph. All linear combinations of strong (stable) pointed graphs form a subalgebra $\mathscr S$,
which contains
certain interesting covariant differential operators arising from deformation quantization on K\"ahler manifolds (cf. Theorem \ref{thmstar}). The composition formula
in $\mathscr S$ is given by
\begin{multline}\label{eqt15}
\left(\sum_{Z_1\in \mathcal G_1}^{\text{strong}}c_1(Z_1)\frac{(-1)^{|V(Z_1)|}}{|{\rm Aut(Z_1)}|}Z_1\right) \circ
\left(\sum_{Z_2\in \mathcal G_1}^{\text{strong}}c_2(Z_2)\frac{(-1)^{|V(Z_2)|}}{|{\rm Aut(Z_2)}|}Z_2\right)\\
=\sum_{Z\in \mathcal G_1}^{\text{strong}}\left(\sum_{\Gamma\subset Z}^{\text{strong}}(-1)^{|V((Z/\Gamma)^s)|+|V(\Gamma^s)|}
c_1((Z/\Gamma)^s)c_2(\Gamma^s)\right)\frac{1}{|{\rm Aut(Z)}|}Z,
\end{multline}
where $\Gamma$ runs over all strong pointed subgraph of $Z$.

Recall that the Berezin transform has an asymptotic expansion (cf. \cite{Eng, KS}),
\begin{equation} \label{eqber}
I_\alpha f(x)=\sum^\infty_{k=0} Q_k f(x)\alpha^{-k},\quad \alpha\rightarrow\infty.
\end{equation}
The following explicit formula for the differential operators $Q_k$ was proved in \cite{Xu2},
\begin{equation}\label{eqQ}
Q_k=\sum_{\Gamma\in \mathcal G_1(k)}^{\text{strong}}\frac{\det(A(\Gamma_-)-I)}{|{\rm Aut(\Gamma)}|}\,\Gamma,
\end{equation}
where $\Gamma_-$ is obtained from $\Gamma$ by removing the distinguished vertex from
$\Gamma$.

We can also study $\mathscr R$ and $\mathscr S$ on a fixed K\"ahler manifold. For a bounded symmetric domain $\Omega$ of rank $r$ equipped
with the Bergman metric, it is
obvious that $\mathscr R=\mathscr S$. Denote by $\mathscr D(\Omega)$ the algebra of invariant differential operators. Engli\v s proved that
$\mathscr S$ coincides with $\mathscr D(\Omega)$ \cite[Prop. 7]{Eng} and $\mathscr S$ is freely generated by $Q_1,Q_3,Q_5,\dots,Q_{2r-1}$ \cite[Thm. 1.1]{Eng0}.

On a bounded symmetric domain, all $R_{i\bar j k\bar l/\alpha}=0,\,|\alpha|\geq1$. So a pointed graph $\Gamma=0$ unless $\Gamma$ is a balanced graph,
i.e. $\deg^+(v)=\deg^-(v)$ for each vertex $v$.
Combining Engli\v s' result and \eqref{eqQ}, we get a set of explicit generators for $\mathscr D(\Omega)$
in terms of balanced strong pointed graphs,
\begin{equation}\label{eqQ2}
Q_k=\sum_{\Gamma\in \mathcal G_1(k)}^{\text{balanced}\atop\text{strong}}\frac{\det(A(\Gamma_-)-I)}{|{\rm Aut(\Gamma)}|}\,\Gamma,\qquad k=1,3,\dots,2r-1,
\end{equation}
whose composition formula is given by \eqref{eqt15}.
Note that on a bounded symmetric domain, balanced strong pointed graphs in $\mathcal G_1(k)$ are not linearly independent. For $k=1,3,5$, $G_k$ has $1,5,119$ nonzero terms respectively.

\vskip 30pt
\section{Star products} \label{secstar}
On a Kahler manifold $(M,\omega_{-1})$, a formal deformation of the form $(1/\nu)\omega_{-1}$ is a formal
$(1,1)$-form,
\begin{equation}\label{eqomega}
  \hat\omega = \frac{1}{\nu} \omega_{-1} + \omega_0 + \nu\omega_1 + \nu^2\omega_2
  + \cdots,
\end{equation}
where each $\omega_k$ is a closed, may be degenerate, $(1,1)$-form. Karabegov \cite{Kar} showed that deformation quantizations with separation of variables on
$(M,\omega_{-1})$ are in one-to-one correspondence with such formal
deformations.
Given a star product $\star$ of anti-Wick type, its Karabegov form is computed as following:
Let $z^1, \ldots, z^n$ be local holomorphic coordinates on
an open subset $U$ of $M$. Then there exists a set of formal functions on $U$, denoted by
$u^1, \ldots, u^n$,
\begin{align*}
  u^k = \frac{1}{\nu} u^k_{-1} + u^k_0 + \nu u^k_1  + \nu^2 u^k_2
 + \cdots,
\end{align*}
satisfying
$u^k \star z^l - z^l \star u^k = \delta^{kl}$.
The Karabegov form of $\star$, which is independent of the coordinates chosen, is given by $\hat\omega \vert_U = - \sqrt{-1}\,\bar\partial (\sum_k u^k dz^k)$.

Let $G=(V,E)$ be a digraph that can be
obtained by inserting a finite number of vertices to edges of a semistable graph $G^{ss}$.
Similar to Definition \ref{dfGS}, we may call such $G$ a \emph{generalized stabilizable graph} if
$G^{ss}$ is stabilizable. The stabilization graph of $G^{ss}$, denoted by $G^s$, is also called the stabilization graph of $G$.

By Lemma \ref{t5} and Lemma \ref{t6}, we know that any strong digraph must be one of the
following: (i) A generalized stabilizable graph; (ii) A single vertex without loops; (iii) A connected linear digraph (i.e. a directed cycle with $n\geq1$ vertices).

Let $h$ be an arbitrary $\mathbb C$-valued function on the set of strong stable graphs and $\{\xymatrix{*+[o][F-]{1}}\}$. We define a function $\alpha_h$ on
the set of all strong digraphs by
\begin{equation*}
\alpha_h(G)=\begin{cases} (-1)^{|V(G)|-|V(G^s)|} h(G^s), & G \mbox{ is a generalized stabilizable graph,}
\\
-1, & G \mbox{ is a single vertex without loops,}\\
(-1)^{n+1}h(\xymatrix{*+[o][F-]{1}}), & G \mbox{ is a directed cycle with } n\geq1 \mbox{ vertices.}
\end{cases}
\end{equation*}

\begin{theorem}\label{thmstar} Let $h$ and $\alpha_h$ be the functions given above. For any functions $f_1$ and $f_2$ on a K\"ahler manifold,
we have the following anti-Wick type star product
\begin{equation} \label{eqstar}
f_1\star f_2(x)=\sum_{\Gamma\in\mathcal G_1^{ss}}^{strong}
\nu^{w(\Gamma)}\frac{1}{|{\rm Aut}(\Gamma)|}\prod_{G\in SCC(\Gamma_-)}\alpha_h(G)
\,\Gamma^{op}(f_1,f_2),
\end{equation}
where $G$ runs over all strongly connected components of $\Gamma_-$ and the partition function $\Gamma^{op}(f_1,f_2)$ is obtained by
taking antiholomorphic and holomorphic derivatives of $\Gamma$ separately on $f_1$ and $f_2$.

The Karabegov form of the star product \eqref{eqstar}
is given by
\begin{equation} \label{eqform}
\hat\omega=\frac1\nu\omega_{-1}-h(\xymatrix{*+[o][F-]{1}})Ric
-\sqrt{-1}\partial\bar\partial\sum_{G\in\mathcal G^{ss}}^{strong}\nu^{w(G)}\frac{\alpha_h(G)}{|{\rm Aut}(G)|}G,
\end{equation}
where $Ric=\sqrt{-1}\partial\bar\partial\log\det g$ is the Ricci curvature.
\end{theorem}
\begin{proof}
The first two terms of formal Berezin transform corresponding to \eqref{eqstar} is
$$I(f)=f+\Big[\xymatrix{\bullet
              \ar@(ur,dr)[]^{1}}\Big]+\cdots,$$
which implies that $\star$ satisfies \eqref{eqone}. The associativity can be verified by
the same argument as \cite[Prop. 4.5]{Xu3}. By definition, in order to prove \eqref{eqform},
we need only check that
\begin{equation}
u^k=\frac1\nu\frac{\partial \Phi}{\partial z^k}-h(\xymatrix{*+[o][F-]{1}})\frac{\partial\log\det g}{\partial z^k}
+\sum_{G\in\mathcal G_0^{ss}}\nu^{w(G)}\, \frac{-\alpha_h(G)}{|{\rm Aut}(G)|}\,\frac{\partial G}{\partial z^k}
\end{equation}
satisfy $u^k\star z^l-z^l u^k=\delta^{kl}$ for $1\leq k,l\leq n$.
The coefficient of $\nu^0$ in $u^k\star_B z^l-z^l u^k$ is equal to
$$\bigg[\xymatrix{\bullet \ar@(ur,dr)[]^{1}}\bigg]^{op}\left(\frac{\partial \Phi}{\partial z^k},z^l\right)
=\frac{\partial^2\Phi}{\partial z^k\partial\bar z^l}=\delta^{kl}.$$

In general, a graph $H$ appearing in $u^k\star z^l-z^l u^k$ has the following form
$$H=\quad\begin{tabular}{c}\xymatrix@C=2mm@R=7mm{ &\ar[dl]_{\bar l} &&\\
 *+++[o][F-]{\Gamma} \ar@/^/@{->} @<1pt> [rrr] \ar
@{-->}[rrr] \ar@/_/@{->} @<-1pt> [rrr] &&& *++[o][F-]{G} \ar[ul]_k}\end{tabular},$$
where $G$ is a strong graph.
It may either come from $\dot H^{op}(\frac{\partial\Phi}{\partial z^k},z^l)$ or $\Gamma^{op}(\frac{\partial G}{\partial z^k},z^l)$,
where $\dot H$ is obtained
from $H$ by gluing the head of $k$ and the tail of $\bar l$.
The coefficient of $H$ in $u^k\star z^l-z^l u^k$ is equal to
$$\prod_{K\in SCC(\dot H_-)}\alpha_h(K)+(-\alpha_h(G))\prod_{K\in SCC(\Gamma_-)}\alpha_h(K)=0,$$
as claimed.
\end{proof}

By specializing the functions $h$ and $\alpha_h$, the above theorem recovers previous known star products: Berezin, Berezin-Toeplitz, Karabegov-Bordemann-Waldmann (standard)
and its dual. For example, take $h(G)=\det(A(G)-I)$ for stable graphs, then $\alpha_h(G)=\det(A(G)-I)$ for all strong graphs (cf. \cite[Lem. 3.9]{Xu3}).
We get the Berezin star product.

From Lemma \ref{t7}, Corollary \ref{t10} and the discussion in \cite[Rem. 4.2]{Xu3}, once \eqref{eqform} is given, the above theorem may be regarded as a special case of Gammelgaard's formula \cite{Gam}.

\begin{proposition} Let $C$ be a constant and $h_C$ be the function given by
\begin{equation}
h_C(G)=\sum_{L\in\mathscr L(G)} \frac{(-1)^{|V(G)|+1+p(L)}C^{p(L)}}{|{\rm Aut}(G)|},
\end{equation}
where $L$ runs over linear subgraphs (including empty subgraph) of $G$ and $p(L)$ the number of components of $L\in\mathscr L(H)$.
Then the corresponding star product \eqref{eqstar} is
\begin{equation} \label{eqstar2}
f_1\star f_2(x)=\sum_{\Gamma\in\mathcal G_1^{ss}}^{strong}
\frac{\nu^{w(\Gamma)}}{|{\rm Aut}(\Gamma)|}\sum_{L\in\mathscr L(\Gamma_-)} \frac{(-1)^{|V(\Gamma)|-1+p(L)}C^{p(L)}}{|{\rm Aut}(\Gamma)|}
\,\Gamma^{op}(f_1,f_2).
\end{equation}
The dual opposite of \eqref{eqstar2} is a star product of Wick type given by
\begin{equation} \label{eqBT}
f_1\star' f_2(x)=\sum_{\Gamma}\nu^{w(\Gamma)}\frac{(-1)^{|E(\Gamma)|}C^{\ell(\Gamma)}}{|{\rm Aut}(\Gamma)|}
\Gamma(f_1,f_2),
\end{equation}
where $\Gamma$ runs over all semistable pointed graphs such that each SCC of $\Gamma_-$ is either a single vertex
or a linear digraph, $\ell(\Gamma)$ is the number of linear digraphs in the SCC's of $\Gamma_-$ and $\Gamma(f_1,f_2)=\Gamma^{op}(f_2,f_1)$.
\end{proposition}
\begin{proof}
The proof of \eqref{eqstar2} is obvious. The proof of \eqref{eqBT} is similar to the argument of \cite[Thm. 4.3]{Xu3}.
\end{proof}

When $C=1$, \eqref{eqstar} and \eqref{eqBT} are respectively the Berezin and Berezin-Toeplitz star products (cf. \cite[\S 4]{Xu3}).
When $C=0$, \eqref{eqstar} and \eqref{eqBT} are respectively the Karabegov-Bordemann-Waldmann star product and its dual (cf. \cite[\S 6]{Xu3}).

\appendix

\vskip 30pt
\section{Enumeration of contractible semistable trees} \label{ap}
\begin{lemma} Let $k,m\geq2$.
Then $t_{k,m}(n)=0$ when $n>k+m-2$ and
\begin{gather}
t_{k,m}(n)=t_{m,k}(n),\qquad t_{k,2}(k)=(2k-3)!!,\\
t_{k,m}(2)=2^k+2^m-k-m-3,\\ \label{eqt10}
t_{k,m}(3)=\frac{1}{2}(3^{k+1}+3^{m+1})-(2^k+2^m)(k+m)-5(2^k+2^m)\\
+2^{k+m}+\frac{1}{2}(k^2+m^2)+\frac{7}{2}(k+m)+km+7.\nonumber
\end{gather}
\end{lemma}
\begin{proof}
The first two equations are obvious. Let us prove $t_{k,2}(k)=(2k-3)!!$. When $k=2$, we have $t_{2,2}(2)=1$. A $k$-vertex tree in $\mathscr T_g(a_1\cdots a_k | \bar b_1\bar b_2)$
can be obtained by connecting the outward leg $a_k$ to a new node in the middle of any of the edges and outward legs of a $(k-1)$-vertex tree
in $\mathscr T_g(a_1\cdots a_{k-1} | \bar b_1\bar b_2)$. There are $k-2$ edges and $k-1$ outward legs in a $(k-1)$-vertex tree
in $\mathscr T_g(a_1\cdots a_{k-1} | \bar b_1\bar b_2)$; therefore, $t_{k,2}(k)$ is larger than $ t_{k-1,2}(k-1)$ by a factor of $2k-3$.
So we get $t_{k,2}(k)=(2k-3)!!$.

There is a unique $2$-vertex tree, so we have
$$t_{k,m}(2)=\sum_{i=0}^{k-2}\binom{k}{i}+\sum_{i=1}^{m-2}\binom{m}{i}=2^k+2^m-k-m-3.$$

There are three $3$-vertex directed trees,
\begin{equation}\label{eqt9}
\xymatrix@C=4mm@R=3mm{ \circ \ar[r] & \circ \ar[r] & \circ}\quad
\xymatrix@C=4mm@R=3mm{ \circ \ar[r] & \circ & \circ \ar[l]}\quad
\xymatrix@C=4mm@R=3mm{ \circ & \circ \ar[l]\ar[r] & \circ}
\end{equation}
We compute their respective contributions to $t_{k,m}(3)$,
\begin{gather*}
t_{k,m}(3)=\sum^{k-3}_{i=1}\binom{k}{i}\sum^{k-2-i}_{j=1}\binom{k-i}{j}+\sum_{i=1}^{k-2}\binom{k}{i}\sum_{j=0}^{m-2}\binom{m}{j}
+\sum_{i=0}^{m-3}\binom{m}{i}\sum_{j=1}^{m-2-i}\binom{m-i}{j}\\
+\frac{1}{2}\sum_{i=2}^{m-2}\binom{m}{i}\sum^{m-i}_{j=2}\binom{m-i}{j}+\frac{1}{2}\sum^{k-2}_{i=2}\binom{k}{i}\sum_{j=2}^{k-i}\binom{k-i}{j},
\end{gather*}
where the first three summations come from the first tree of \eqref{eqt9}, the last two summations come from the second
and third trees of \eqref{eqt9} respectively. We can simplify the binomials to get \eqref{eqt10}.
\end{proof}

\begin{table}[h]
\caption{$t_{k,m}(n)$, numbers of contractible semistable trees} \label{tb1}
\begin{tabular}{|c||c|c|c|c|c|c|}
\hline
    $(k,m)$                                            &$n=1$&$n=2$&$n=3$&$n=4$ &$n=5$  &$n=6$
\\\hline       $(2,2)$      &$1$&$1$& &  &  &
\\\hline       $(3,2)$                  &$1$&$4$&$3$&  &  &
\\\hline       $(4,2)$                  &$1$&$11$&$25$&$15$ &  &
\\\hline       $(3,3)$                  &$1$&$7$&$15$&$9$ &  &
\\\hline       $(5,2)$                  &$1$&$26$&$130$&$210$  &$105$  &
\\\hline       $(4,3)$                  &$1$&$14$&$58$&$90$ & $45$ &
\\\hline       $(6,2)$                  &$1$&$57$&$546$&$1750$ &$2205$  & $945$
\\\hline       $(5,3)$                  &$1$&$29$&$208$&$628$ & $765$ &$325$
\\\hline       $(4,4)$                  &$1$&$21$&$150$&$432$ &$529$  & $225$
\\\hline
\end{tabular}
\end{table}

$$ \ \ \ \ $$

\

\end{document}